\documentclass{article}
\usepackage{amsfonts}
\usepackage{amsmath}

\newtheorem{theorem}{Theorem}

\newtheorem{corollary}[theorem]{Corollary}

\newtheorem{example}[theorem]{Example}

\newtheorem{lemma}[theorem]{Lemma}

\newtheorem{proposition}[theorem]{Proposition}
\newtheorem{remark}[theorem]{Remark}

\newenvironment{proof}[1][Proof]{\textbf{#1.} }{\ \rule{0.5em}{0.5em}}
\input{tcilatex}
\begin{document}

\author{S.Barannikov \\
ENS(Paris). Universit\'{e} Paris Diderot - Paris 7.}
\title{Matrix De Rham complex and quantum $A-$infinity algebras.}
\date{}
\maketitle

\begin{abstract}
I establish the relation of the non-commutative BV-formalism with
super-invariant matrix integration. In particular, the non-commutative
BV-equation, defining the quantum A-infinity-algebras, introduced in \cite%
{B1}, is represented via de Rham differential acting on the supermatrix
spaces related with Bernstein-Leites simple associative algebras with odd
trace $q(N)$, and with $gl(N|N)$. I also show that the Lagrangians of the $EA
$-matrix integrals from \cite{B2} are represented by equivariantly closed
differential forms.
\end{abstract}

The \footnotetext[1]{%
submitted for publication on 20/01/2010, preprint HAL-00378776, (04/2009).
Published in Lett. in Math.Phys. 10.1007/s11005-013-0677-7
\par
{}} relation of the cyclic differential with $gl(N)-$invariant tensors on
matrix spaces $gl(N)\otimes \Pi V$ was one of the origins of the cyclic
homology \cite{FT},\cite{L}. Motivated by the ideas of supersymmetry, in
particular the importance of the $\emph{odd}$ symplectic and
Batalin--Vilkovisky structures for construction of Lagrangians of physical
theories, I propose to include into consideration other series of simple
associative \emph{super} algebras, such as Bernstein-Leites algebra $q(N)$ 
\cite{BL}, which is the odd analogue of the general linear matrix algebra.

I start by establishing the relation of the non-commutative BV-formalism,
introduced in \cite{B1}, with invariant super-matrix integration. In
particular I express the non-commutative Batalin-Vilkovisky equation,
defining the quantum $A_{\infty }$-algebras, via the de Rham differential
acting on the super-matrix spaces 
\begin{equation*}
q(N)\otimes \Pi V
\end{equation*}%
constructed from the Bernstein-Leites algebra in the even scalar product
case, and on the super-matrix spaces 
\begin{equation*}
gl(N|N)\otimes \Pi V
\end{equation*}%
in the odd scalar product case. It implies, in particular, that the
cohomology of the Batalin-Vilkovisky differential from loc.cit. are zero. As
another immediate consequence I prove that the Lagrangians of the
supersymmetric matrix integrals, introduced in \cite{B2}, represent closed%
\emph{\ }$gl(N)$-\emph{equivariant} differential forms. In the even scalar
product case the Lagrangian is 
\begin{equation}
\exp \frac{1}{\hbar }\left( -\frac{1}{2}\left\langle [\Xi ,X],X\right\rangle
+S_{q}(X)\right)   \label{lagrr1}
\end{equation}%
where $\Xi \in q(N)_{1}$ is an odd element of the Lie algebra $q(N)$, and 
\begin{equation*}
S_{q}(X)=\sum_{\alpha ,g,i}\hbar ^{2g-1+i}c_{g,\alpha _{1}\ldots \alpha
_{i}}otr(X^{\alpha _{1}})\ldots otr(X^{\alpha _{i}})
\end{equation*}%
is $q(N)-$invariant function associated by the invariant theory with a sum
of \emph{exterior} products of cyclic cochains, representing a solution to
the noncommutative BV-equation, and $otr(Y)$ is the odd trace on $q(N)$, the
odd analogue of the matrix algebra. In the odd scalar product case the
Lagrangian is 
\begin{equation}
\exp \frac{1}{\hbar }\left( -\frac{1}{2}\left\langle [\Xi ,X],X\right\rangle
+S_{gl}(X)\right)   \label{lagrr2}
\end{equation}%
where $\Xi \in gl(N|N)_{1}$ is an odd element, and 
\begin{equation*}
S_{gl}(X)=\sum_{\alpha ,g,i}\hbar ^{2g-1+i}c_{g,\alpha _{1}\ldots \alpha
_{i}}tr(X^{\alpha _{1}})\ldots tr(X^{\alpha _{i}})
\end{equation*}%
is $gl(N|N)-$invariant function associated by the invariant theory with a
sum of \emph{symmetric} products of cyclic cochains, representing a solution
to the noncommutative BV-equation in the odd scalar product case, and the
trce is the supertrace on $gl(N|N)$.

As another consequence of the relation of non-commutative BV-formalism with
super-invariant matrix integration, I prove an analogue of Morita
equivalence for non-commutative BV-formalism and construct from solutions of
the non-commutative BV-equations corresponding to vector space with scalar
product $(V,l)$, the new solutions corresponding to the vector spaces $%
(V\otimes gl(k|\widetilde{k}),l\otimes tr)$ and $(V\otimes q(N),l\otimes
otr) $.

The integrals of (\ref{lagrr1}), and their odd dimensional counterparts (\ref%
{lagrr2}), are the higher dimensional generalisations of the matrix Airy
function. The case of the matrix Airy function corresponds to the simplest,
zero-dimensional, solution, associated with the algebra generated by the
identity element: $e\cdot e=e$.

The results of the current paper suggest that, in a sense, any topological
matrix Lagrangian, with symmetry algebras given by the super Lie algebras $%
q(N)$ or $gl(N|N)$, comes from the construction from \cite{B2}.

The study of integrals of (\ref{lagrr1}),(\ref{lagrr2}) can be viewed as
generalisation for families depending on non-commuting parameters (i.e.
points over non-commutative super algebras) of the study of the $A_{\infty
}- $periods from \cite{B0}: 
\begin{equation*}
\int (\exp \frac{1}{\hbar }\gamma _{A_{\infty }})\omega ~
\end{equation*}%
where $\gamma _{A_{\infty }}$ represents the $A_{\infty }$- structure
depending on commuting parameters. For such families of integrals there is
an analogue, in the non-commutative setting, of the theory of Hodge
structures and their variations, which involves the notion of \emph{%
semi-infinite} Hodge structures, see loc.cit. The developement of an
analogue of the theory of variations of Hodge structures in the setting of
families depending on noncommuting parameters is an interesting open problem
to which I plan to return in the future publications.

Here is the short description of sections of the paper. In the first two
sections the $gl(N|N)-$invariant geometry of the odd symplectic affine space 
$gl(N|N)\otimes \Pi V$ is studied and the noncommutative Batalin-Vilkovisky
differential on $Symm(\oplus _{j=0}^{\infty }(V^{\otimes j})^{\mathbb{Z}/j%
\mathbb{Z}})$ from \cite{B1} is identified with $gl(N|N)-$invariant
BV-differential on this affine space. In the next section I prove analogous
result in the even scalar product case, corresponding to the noncommutative
BV-differential on the \emph{exterior} product of cyclic chains $Symm(\oplus
_{j=0}^{\infty }\Pi ((\Pi V)^{\otimes j})^{\mathbb{Z}/j\mathbb{Z}})$ and the 
$q(N)-$invariant BV-differential on the odd symplectic affine space $%
q(N)\otimes \Pi V$. The main technical result used here is the analogue of
the invariant theory for the algebra $q(N)$ from \cite{S}. By the standard
odd Fourier transform the BV-differentials on affine matrix spaces are
identified with de Rham differentials on the similar affine matrix spaces.
This implies the triviality of the cohomology of the noncommutative
Batalin-Vilkovisky differentials on $Symm(\oplus _{j=0}^{\infty }(V^{\otimes
j})^{\mathbb{Z}/j\mathbb{Z}})$ and $Symm(\oplus _{j=0}^{\infty }\Pi ((\Pi
V)^{\otimes j})^{\mathbb{Z}/j\mathbb{Z}})$. I also prove a kind of Morita
equivalence, which, given solutions to the noncommutative BV equation
corresponding to a space $V$ , allows one to construct solutions
corresponding to the spaces $V\otimes gl(k|l)$ and $V\otimes q(N)$. In the
section \ref{secEqDif} the hamiltonians for the adjoint actions of super Lie
algebras $q(N)$ and $gl(N|N)$ are written and it is shown that their sum
with BV-differential corresponds under the odd Fourier transform, to the
equivariant cohomology differential. The section \ref{secNCaksz} is devoted
to noncommutative \emph{super}-equivariant AKSZ-type symplectic $\sigma -$%
model interpretation of the Lagrangians from \cite{B2}. It is the invariance
with respect to the \emph{supergroups} $GQ(N)$ and $GL(N|N)$ that plays the
essentual role in this interpretation.

I would like to thank the referees for useful comments.

\emph{Notations}. I work in the tensor category of $\mathbb{Z}/2\mathbb{Z}$%
-graded vector spaces, over an algebraically closed field $k$, $char(k)=0$.
Let $V=V_{0}\oplus V_{1}$ be a $\mathbb{Z}/2\mathbb{Z}$-graded vector space.
I denote by $\overline{\alpha }$ the parity of an element $\alpha $ and by $%
\Pi V$ the super vector space with inversed parity. For a finite group $G$
acting on a $\mathbb{Z}/2\mathbb{Z}$-graded vector space $U$, I denote via $%
U^{G}$ the space of invariants with respect to the action of $G$ and by $%
U_{G}$ the space of coinvariants $U_{G}=U/\{gv-v|v\in V,g\in G\}$. If $G$ is
finite then the averaging $(v)\rightarrow 1/|G|\sum_{g\in G}gv$ give a
canonical isomorphism $U_{G}\simeq U^{G}$. Element $(a_{1}\otimes
a_{2}\otimes \ldots \otimes a_{n})$ of $A^{\otimes n}$ is denoted by $%
(a_{1},a_{2},\ldots ,a_{n})$. Cyclic words, i.e. elements of the subspace $%
(V^{\otimes n})^{\mathbb{Z}/n\mathbb{Z}}$ are denoted via $(a_{1}\ldots
a_{n})^{c}$. The symbol $\delta _{\alpha }^{\beta }$ denotes the Kronecker
delta tensor: $\delta _{\alpha }^{\beta }=1$ for $\alpha =\beta $ and zero
otherwise. I denote by $tr$ the super trace linear functional on $End(U)$, $%
tr(U)=\sum_{a}(-1)^{\overline{a}}U_{a}^{a}$. The isomorphism of the tensor
category of $\mathbb{Z}/2\mathbb{Z}$-graded vector spaces, $X\otimes Y$ $%
\simeq Y\otimes X$, is realized via $(x,y)\rightarrow (-1)^{\overline{x}%
\overline{y}}(y,x)$. Throughout the paper, unless it is stated explicitly
otherwise, $(-1)^{\epsilon _{K}}$ in the formulas denotes the standard
Koszul sign, which can be worked out by counting $(-1)^{\overline{a}%
\overline{b}}$every time the objects $a$ and $b$ are interchanged to obtain
the given formula.

\section{The vector space $F$.}

\bigskip I start with the case of the \emph{odd }symmetric scalar product on 
$\mathbb{Z}/2\mathbb{Z}$-graded vector space $V$. In this case I put $%
F=\oplus _{n=0}^{\infty }F_{n}$ where%
\begin{equation*}
F_{n}=(V^{\otimes n}\otimes k[\mathbb{S}_{n}])^{\mathbb{S}_{n}}.
\end{equation*}%
Here, $k\mathbf{[}\mathbb{S}_{n}]$ is the group algebra of the symmetric
group $\mathbb{S}_{n}$, and $\mathbb{S}_{n}$ acts on $k[\mathbb{S}_{n}]$ by
conjugation.

The odd symmetric scalar product on $V$ defines a differential on $F$, \cite%
{B1}, which equips $F$ with the Batalin-Vilkovisky algebra structure. My aim
below is to use the invariant theory approach to cyclic homology (\cite{FT},%
\cite{L} and references therein) in order to represent this differential on $%
F$ via $gl(N|N)-$invariant geometry on the affine spaces $gl(N|N)\otimes \Pi
V$ . First, I start by interpreting $F$ as the space of $gl(N|\widetilde{N})$%
-invariant symmetric tensors on $gl(N|\widetilde{N})\otimes \Pi V$.

Let $U$ be a $\mathbb{Z}/2\mathbb{Z}$-graded vector space. There is the
natural left group $\mathbb{S}_{n}$-action on $U^{\otimes n}$ via 
\begin{equation*}
\sigma \in \mathbb{S}_{n},\,\,\,\sigma :(u_{1},\ldots ,u_{n})\rightarrow
(-1)^{\epsilon _{K}}(u_{\sigma ^{-1}(1)},\ldots ,u_{\sigma ^{-1}(n)})
\end{equation*}%
where $(-1)^{\epsilon _{K}}$ is the standard Koszul sign, which is equal in
this case to 
\begin{equation*}
\sum_{i<j,\sigma (i)>\sigma (j)}(-1)^{\overline{u}_{i}\overline{u}_{j}}
\end{equation*}%
This gives $k-$algebra morphism $\mu :k[\mathbb{S}_{n}]\rightarrow
End_{k}(U^{\otimes n})$. The group $GL(U)$ of automorphisms of $U$ acts
diagonally on $U^{\otimes n}$ and the image of $k[\mathbb{S}_{n}]$ is in the
invariant subspace of the adjoint action of $GL(U)$ on $End_{k}(U^{\otimes
n})$. If $U$ is a vector space with $\dim _{k}U_{0}\geq n$ then the $k-$%
algebra morphism $\mu $ is an isomorphism: 
\begin{equation}
\mu :k[\mathbb{S}_{n}]\simeq (End_{k}(U^{\otimes n}))^{GL(U)}  \label{mukasn}
\end{equation}%
according to the invariant theory, see for example (\cite{L}, 9.1.4),
arguments from which are easily adopted to work in the case of the $\mathbb{Z%
}/2\mathbb{Z}$-graded vector space.

\begin{proposition}
\bigskip \label{propFn}The vector space $F_{n}$ is canonically identified
via the map $\mu $ with $GL(U)$-invariant subspace of $n-$symmetric powers
of the vector space $End_{k}(U)\otimes V$: 
\begin{equation}
F_{n}\simeq \left( S^{n}(End_{k}(U)\otimes V)\right) ^{GL(U)},  \label{glinv}
\end{equation}%
where $U$ is a $\mathbb{Z}/2\mathbb{Z}$-graded vector space with $\dim _{k}$ 
$U_{0}\geq n$.
\end{proposition}

\begin{proof}
The proof is essentially the application of the invariant theory as in the
classical definition of cyclic homology via homology of the general linear
algebra (see \cite{FT},\cite{L}). I have the following sequence of
isomorphisms of $\mathbb{Z}/2\mathbb{Z}$-graded vector spaces:%
\begin{multline*}
(V^{\otimes n}\otimes k[\mathbb{S}_{n}])^{\mathbb{S}_{n}}\simeq (V^{\otimes
n}\otimes (End_{k}(U^{\otimes n}))^{GL(U)})^{\mathbb{S}_{n}}\simeq  \\
\simeq (V^{\otimes n}\otimes (End_{k}(U)^{\otimes n})^{GL(U)})^{\mathbb{S}%
_{n}}\simeq ((V^{\otimes n}\otimes End_{k}(U)^{\otimes n})^{GL(U)})^{\mathbb{%
S}_{n}}\simeq  \\
\simeq (((V\otimes End_{k}(U))^{\otimes n})^{\mathbb{S}_{n}})^{GL(U)}\simeq
\left( S^{n}(End_{k}(U)\otimes V)\right) ^{GL(U)}.
\end{multline*}%
Here, I used the canonical isomorphism $End_{k}(U)^{\otimes n}\simeq
End_{k}(U^{\otimes n})$, under which the permuting of $n-$tuples of
endomorphisms by $\sigma $ corresponds to the conjugation by $\mu (\sigma )$%
. I also used the fact that $GL(U)-$action and $\mathbb{S}_{n}-$action
mutually commute.
\end{proof}

I shall denote the isomorphism (\ref{glinv}) by $\mu _{F}$. Denote by $%
\{E_{\alpha }^{\beta }\}$, $E_{\alpha }^{\beta }=e_{\alpha }^{\ast }\otimes
e_{\beta }$ the basis of elementary matrices in $End_{k}(U)$ corresponding
to some basis $\{e_{\alpha }\}$ in $U$. Then the map (\ref{mukasn}) is
written as 
\begin{equation*}
\mu :[\sigma ]\rightarrow \sum_{\alpha _{1},\ldots \alpha
_{n}}(-1)^{\epsilon _{K}}E_{\alpha _{1}}^{\alpha _{\sigma ^{-1}(1)}}\otimes
\ldots \otimes E_{\alpha _{n}}^{\alpha _{\sigma ^{-1}(n)}}.
\end{equation*}%
For a set of elements $a_{i}\in V$ , $i\in \{1,\ldots ,n\}$ denote via%
\begin{equation}
A_{i,\alpha }^{\beta }=E_{\alpha }^{\beta }\otimes a_{i},\,~\,\,A_{i,\alpha
}^{\beta }\in End_{k}(U)\otimes V  \label{Ailb}
\end{equation}%
the corresponding set of generators of the symmetric algebra of $%
End_{k}(U)\otimes V$. Denote 
\begin{equation}
Tr(A_{\rho _{1}}\ldots A_{\rho _{r}})=\sum_{\alpha _{1},\ldots \alpha
_{r}}(-1)^{\epsilon _{K}}A_{\rho _{1},\alpha _{1}}^{\alpha _{r}}\cdot
A_{\rho _{2},\alpha _{2}}^{\alpha _{1}}\ldots \cdot A_{\rho _{r},\alpha
_{r}}^{\alpha _{r-1}}  \label{Trt}
\end{equation}%
the symmetric tensor of degree $r$.

\begin{remark}
The notations (\ref{Ailb}), (\ref{Trt}) are justified by the fact that if
one identifes $End_{k}(U)$ with its dual space $\limfunc{Hom}(End_{k}(U),k))$
using the super trace and therefore, identifies $S^{n}(End_{k}(U)\otimes V)$
with polynomial functions of degree $n$ on the vector space $\limfunc{Hom}%
(V,End_{k}(U))$, so that the isomorphism $\mu _{F}$ becomes 
\begin{equation*}
F_{n}\simeq \left( S^{n}(\limfunc{Hom}(\limfunc{Hom}(V,End_{k}(U)),k)\right)
)^{GL(U)}
\end{equation*}%
then $A_{i,\alpha }^{\beta }$ is the linear function on $\limfunc{Hom}%
(V,End_{k}(U))$: $\varphi \rightarrow \varphi (a_{i})_{\alpha }^{\beta }$,
where $\varphi (a_{i})e_{\alpha }=\varphi (a_{i})_{\alpha }^{\beta }e_{\beta
}$, and (\ref{Trt}) is the super trace of the action of product of matrices $%
\varphi (a_{\rho _{1}})\ldots \varphi (a_{\rho _{r}})$ on $U$.
\end{remark}

The space $F_{n}$ is generated linearly by $\mathbb{S}_{n}-$invariant
elements of the form 
\begin{equation}
(a_{\rho _{1}}\ldots a_{\rho _{r}})^{c}\cdot \ldots \cdot (a_{\tau
_{1}}\ldots a_{\tau _{t}})^{c}\otimes \sigma  \label{aaaa}
\end{equation}%
where $\sigma =(\rho _{1}\ldots \rho _{r})\ldots (\tau _{1}\ldots \tau _{t})$
is the cycle decomposition of a permutation $\sigma \in \mathbb{S}_{n}$.

\begin{lemma}
The isomorphism (\ref{glinv}) sends the element (\ref{aaaa}) to the $GL(U)$%
-invariant symmetric tensor%
\begin{equation*}
Tr(A_{\rho _{1}}\ldots A_{\rho _{r}})\cdot \ldots \cdot Tr(A_{\tau
_{1}}\ldots A_{\tau _{t}}).
\end{equation*}
\end{lemma}

\begin{proof}
By definition $\mu _{F}$ sends such element to 
\begin{equation}
\sum_{\alpha _{1},\ldots \alpha _{n}}(-1)^{\epsilon _{K}}A_{1,\alpha
_{1}}^{\alpha _{\sigma ^{-1}(1)}}\cdot \ldots \cdot A_{n,\alpha
_{n}}^{\alpha _{\sigma ^{-1}(n)}}.  \label{asigma-1}
\end{equation}%
It is sufficient now to rearrange (\ref{asigma-1}), so that the pairs of
terms with the same repeating upper and lower indexes are placed one after
the other.
\end{proof}

\section{The BV-differential and the bracket.}

Let us now assume that $V$ has an \emph{odd} symmetric non-degenerate scalar
product 
\begin{equation*}
l:V^{\otimes 2}\rightarrow \Pi k,l(x,y)=(-1)^{\overline{x}\,\overline{y}%
}l(y,x)
\end{equation*}%
It follows in particular that the even and odd components of $V$ are of the
same dimension, $\dim _{k}V=(r|r)$.

I assume from now on that $U$ is the $\mathbb{Z}/2\mathbb{Z}$-graded vector
space which also has even and odd components of the same dimension: 
\begin{equation*}
\dim _{k}U=(N|N).
\end{equation*}%
The super-trace functional 
\begin{equation*}
tr(E_{\alpha }^{\widetilde{\alpha }})=(-1)^{\overline{\alpha }}\delta
_{\alpha }^{\widetilde{\alpha }}
\end{equation*}%
defines the natural even scalar product on the vector space $End_{k}(U)$:%
\begin{equation}
tr(E_{\widetilde{\beta }}^{\widetilde{\alpha }},E_{\alpha }^{\beta })=(-1)^{%
\overline{\alpha }}\delta _{\alpha }^{\widetilde{\alpha }}\delta _{%
\widetilde{\beta }}^{\beta }  \label{trEnd}
\end{equation}%
It allows to extend the odd symmetric scalar product $l$ on $V$ to the odd
symmetric non-degenerate scalar product $\widehat{l}$ on $gl(N|N)\otimes V$.
The latter space is therefore, an affine space with constant odd symplectic
structure. Its algebra of symmetric tensors $\oplus _{n=0}^{\infty
}S^{n}(gl(N|N)\otimes V)$ is naturally a Batalin-Vilkovisky algebra. If I
choose a basis $\{a_{\nu }\}$ in $V$, then the Batalin-Vilkovisky operator
acting on the symmetric algebra $\oplus _{n=0}^{\infty }S^{n}(gl(N|N)\otimes
V)$ is written, using the generators $A_{\nu ,\alpha }^{\beta }$, as%
\begin{equation}
\Delta =\sum_{\nu \kappa ,\alpha \beta }(-1)^{\varepsilon _{K}}\frac{l_{\nu
\kappa }}{2}\frac{\partial ^{2}}{\partial A_{\nu ,\alpha }^{\beta }\partial
A_{\kappa ,\beta }^{\alpha }}  \label{deltaAA}
\end{equation}%
where $l_{\nu \kappa }=l(a_{\nu },a_{\kappa })$, and the Koszul sign in this
case $\varepsilon _{K}=\overline{\beta }+\overline{a}_{\nu }(\overline{%
\alpha }+\overline{\beta })$. Similarly I have the standard odd Poisson
bracket corresponding to the affine space with constant odd symplectic
structure:%
\begin{equation*}
\{A_{\nu ,\alpha }^{\beta }A_{\kappa ,\widetilde{\beta }}^{\widetilde{\alpha 
}}\}=(-1)^{\varepsilon _{K}}\delta _{\alpha }^{\widetilde{\alpha }^{\prime
}}\delta _{\widetilde{\beta }}^{\beta }l_{\nu \kappa }
\end{equation*}%
Since \bigskip the scalar product $\widehat{l}$ is $GL(N|N)$-invariant,
therefore, both the second-order odd operator $\Delta $ and the bracket $%
\{\bullet ,\bullet \}$ are $GL(N|N)$-invariant. Therefore, $\Delta $ defines
a differential on the $GL(N|N)-$invariant subspace 
\begin{equation}
\oplus _{n\leq s}S^{n}(gl(N|N)\otimes V)^{GL(N|N)},  \label{snglu}
\end{equation}%
which coincides with $\oplus _{n\leq s}F^{n}$ by proposition \ref{propFn} if 
$\dim _{k}U$ is sufficiently big ($N\geq s$). This also gives a bracket $%
\{\gamma _{1},\gamma _{2}\}\in \oplus _{n\leq s}S^{n}(gl(N|N)\otimes
V)^{GL(N|N)}$ for elements $\gamma _{i}\in $ $S^{n_{i}}(gl(N|N)\otimes
V)^{GL(N|N)}$, $n_{1}+n_{2}\leq s+2$.

The Batalin-Vilkovisky operator acting on $F$ was introduced in \cite{B1} .
It is the combination of \textquotedblright
dissection-gluing\textquotedblright\ operator acting on cycles with
contracting by the tensor of the scalar product. The space $F$ is naturally
identified, by considering the cycle decomposition of permutations, with the
symmetric algebra of the space of cyclic words: 
\begin{equation*}
F=Symm(\oplus _{j=0}^{\infty }(V^{\otimes j})^{\mathbb{Z}/j\mathbb{Z}}).
\end{equation*}%
The second order Batalin-Vilkovisky operator from (\cite{B1},\cite{B2}) is
completely determined by its action on the second symmetric power and it
sends a product of two cyclic words $(a_{\rho _{1}}\ldots a_{\rho
_{r}})^{c}(a_{\tau _{1}}\ldots a_{\tau _{t}})^{c}$ to 
\begin{multline}
\sum_{p,q}(-1)^{\varepsilon _{1,K}}l_{\rho _{p}\tau _{q}}(a_{\rho
_{1}}\ldots a_{\rho _{p-1}}a_{\tau _{q+1}}\ldots a_{\tau _{q-1}}a_{\rho
_{p+1}}\ldots a_{\rho _{r}})^{c}+  \label{deltas} \\
+\sum_{p+1<q}(-1)^{\varepsilon _{2,K}}l_{\rho _{p}\rho _{q}}(a_{\rho
_{1}}\ldots a_{\rho _{p-1}}a_{\rho _{q+1}}\ldots a_{\rho _{r}})^{c}(a_{\rho
_{p+1}}\ldots a_{\rho _{q-1}})^{c}(a_{\tau _{1}}\ldots a_{\tau _{t}})^{c} 
\notag \\
+\sum_{p+1<q}(-1)^{\varepsilon _{3,K}}l_{\tau _{p}\tau _{q}}(a_{\rho
_{1}}\ldots a_{\rho _{r}})^{c}(a_{\tau _{1}}\ldots a_{\tau _{p-1}}a_{\tau
_{q+1}}\ldots a_{\tau _{t}})^{c}(a_{\tau _{p+1}}\ldots a_{\tau _{q-1}})^{c} 
\notag
\end{multline}%
where $\varepsilon _{i,K}$ are the Koszul signs, and $q^{\prime }<q$ denotes
the cyclic order. It follows from \cite{B1} that $\Delta ^{2}=0$.

\begin{theorem}
\label{deltamx}The operator $\Delta $ defined on the $GL(N|N)-$invariant
subspace (\ref{snglu}), with sufficiently big $N\geq s$, coincides with the
BV-differential defined on $\oplus _{n\leq s}F_{n}$ in \cite{B1},\cite{B2}.
\end{theorem}

\begin{proof}
As $\Delta $ is of the second order with respect to the multiplication, it
is sufficient to consider the case of a product of two cyclic words 
\begin{equation*}
(a_{\rho _{1}}\ldots a_{\rho _{r}})^{c}(a_{\tau _{1}}\ldots a_{\tau
_{t}})^{c}
\end{equation*}%
which corresponds under $\mu _{F}$ to 
\begin{equation*}
Tr(A_{\rho _{1}}\ldots A_{\rho _{r}})Tr(A_{\tau _{1}}\ldots A_{\tau _{t}}).
\end{equation*}%
Applying 
\begin{equation*}
\Delta =\sum_{\nu ,\kappa ;\theta ,\beta }(-1)^{\varepsilon _{K}}\frac{%
l_{\nu \kappa }}{2}\frac{\partial ^{2}}{\partial A_{\nu ,\theta }^{\beta
}\partial A_{\kappa ,\beta }^{\theta }}
\end{equation*}%
I get three terms. First, there is the term 
\begin{multline*}
\sum_{p,q;\theta ,\beta }(-1)^{\varepsilon _{K}}l_{\rho _{p}\tau _{q}}(\sum 
_{\substack{ \alpha _{1},\ldots \alpha _{r};  \\ \alpha _{p-1}=\beta ,\alpha
_{p}=\theta }}(-1)^{\varepsilon _{p,K}}A_{\rho _{1},\alpha _{1}}^{\alpha
_{r}}\ldots \widehat{A_{\rho _{p},\theta }^{\beta }}\ldots A_{\rho
_{r},\alpha _{r}}^{\alpha _{r-1}})\cdot \\
\cdot (\sum_{\substack{ \gamma _{1},\ldots \gamma _{t}  \\ \gamma
_{q-1}=\theta ,\gamma _{q}=\beta }}(-1)^{\epsilon _{q,K}}A_{\tau _{1},\gamma
_{1}}^{\gamma _{t}}\ldots \widehat{A_{\tau _{q},\beta }^{\theta }}\ldots
A_{\tau _{t},\gamma _{t}}^{\gamma _{t-1}})
\end{multline*}%
Rewriting this term so that the pairs of terms with repeating lower and
upper indexes follow one after the other I get%
\begin{eqnarray*}
&&\sum_{p,q;\theta ,\beta }(-1)^{\varepsilon _{1,K}}l_{\rho _{p}\tau
_{q}}\cdot \\
&&\cdot (\sum_{\{\alpha \},\{\gamma \}}(-1)^{\varepsilon _{\rho \tau
,K}}A_{\rho _{1},\alpha _{1}}^{\alpha _{r}}\ldots A_{\rho _{p-1},\beta
}^{\alpha _{p-2}}A_{\tau _{q+1},\gamma _{q+1}}^{\beta }\ldots A_{\tau
_{q-1},\theta }^{\gamma _{q-2}}A_{\rho _{p+1},\alpha _{p+1}}^{\theta }\ldots
A_{\rho _{r},\alpha _{r}}^{\alpha _{r-1}})
\end{eqnarray*}%
where $\{\alpha \}=\{\alpha _{1},\ldots \widehat{\alpha }_{p-1},\widehat{%
\alpha }_{p\ldots }\alpha _{r}\}$, $\{\gamma \}=\{\gamma _{1},\ldots 
\widehat{\gamma }_{q-1},\widehat{\gamma }_{q}\ldots \gamma _{t}\}$, which is 
\begin{equation*}
\sum_{p,q}(-1)^{\varepsilon _{1,K}}l_{\rho _{p}\tau _{q}}Tr(A_{\rho
_{1}}\ldots A_{\rho _{p-1}}A_{\tau _{q+1}}\ldots A_{\tau _{q-1}}A_{\rho
_{p+1}}\ldots A_{\rho _{r}}).
\end{equation*}%
This is the term corresponding to the first term in the definition of the
noncommutative Batalin-Vilkovisky operator acting on the product of the two
cyclic words. The matching of signs is verified by using the general
formalism of the Koszul rule, via multiplying the odd elements by associated
odd parameters and verifying that the signs match in the even degree case.
The second term is 
\begin{multline*}
\sum_{p<q;\theta ,\beta }(-1)^{\varepsilon _{K}}l_{\rho _{p}\rho _{q}}(\sum 
_{\substack{ \alpha _{1},\ldots \alpha _{r};  \\ \alpha _{p-1}=\beta ,\alpha
_{p}=\theta  \\ \alpha _{q-1}=\theta ,\alpha _{q}=\beta }}(-1)^{\epsilon
_{p,K}}A_{\rho _{1},\alpha _{1}}^{\alpha _{r}}\ldots \widehat{A_{\rho
_{p},\theta }^{\beta }}\ldots \widehat{A_{\rho _{q},\beta }^{\theta }}\ldots
A_{\rho _{r},\alpha _{r}}^{\alpha _{r-1}})\cdot \\
\cdot (\sum_{\gamma _{1},\ldots \gamma _{t}}(-1)^{\epsilon _{\tau
,K}}A_{\tau _{1},\gamma _{1}}^{\gamma _{t}}\ldots A_{\tau _{t},\gamma
_{t}}^{\gamma _{t-1}}).
\end{multline*}%
Assume first that the erased terms $\widehat{A_{\rho _{p},\theta }^{\beta }}$
and $\widehat{A_{\rho _{q},\beta }^{\theta }}$ are not sitting next to each
other, i.e. $p+1<q$. Then, rewriting this expression so that the pairs of
terms with the same repeating lower and upper indexes follow one after the
other, gives 
\begin{eqnarray*}
&&\sum_{p+1<q}(-1)^{\varepsilon _{2,K}}l_{\rho _{p}\rho _{q}}\cdot \\
&&\cdot Tr(A_{\rho _{1}}\ldots A_{\rho _{p-1}}A_{\rho _{q+1}}\ldots A_{\rho
_{r}})Tr(A_{\rho _{p+1}}\ldots A_{\rho _{q-1}})Tr(A_{\tau _{1}}\ldots
A_{\tau _{t}}).
\end{eqnarray*}%
This correseponds to the second term in the formula for the noncommutative
Batalin-Vilkovisky operator above. However if $p+1=q$ then I get instead 
\begin{equation*}
\sum_{p}(-1)^{\widetilde{\varepsilon }_{K}}l_{\rho _{p}\rho
_{p+1}}Tr(A_{\rho _{1}}\ldots A_{\rho _{p-1}}A_{\rho _{p+2}}\ldots A_{\rho
_{r}})Tr(Id)Tr(A_{\tau _{1}}\ldots A_{\tau _{t}}),
\end{equation*}%
where $Tr(Id)=\sum_{\alpha }(-1)^{\overline{\alpha }}$ , which is equal to
zero precisely because the even and odd parts of $U$ are of the same
dimension 
\begin{equation}
\dim _{k}U_{0}=\dim _{k}U_{1}.  \label{uevenuodd}
\end{equation}
The third term is similar to the second one and it gives 
\begin{eqnarray*}
&&\sum_{p+1<q}(-1)^{\varepsilon _{3,K}}l_{\tau _{p}\tau _{q}}\cdot \\
&&\cdot Tr(A_{\rho _{1}}\ldots A_{\rho _{r}})Tr(A_{\tau _{1}}\ldots A_{\tau
_{p-1}}A_{\tau _{q+1}}\ldots A_{\tau _{t}})Tr(A_{\tau _{p+1}}\ldots A_{\tau
_{q-1}})
\end{eqnarray*}%
corresponding to the third term. So I get the three terms corresponding
exactly to the noncommutaive Batalin -Vilkovisky operator defined in (\cite%
{B1},\cite{B2})
\end{proof}

\bigskip Since the map $\mu _{F}$ respects the multiplicative structure, the
similar result concerning the odd symplectic bracket follows immediately.

\begin{proposition}
\label{propbrackt}The odd symplectic bracket 
\begin{equation*}
F_{n}\otimes F_{n^{\prime }}\rightarrow F_{n+n^{\prime }-2}
\end{equation*}%
coincides with the odd symplectic bracket on the $GL(N|N)-$invariant
subspaces for sufficiently big $N>n+n^{\prime }$ 
\begin{multline*}
S^{n}(gl(N|N)\otimes V)^{GL(N|N)}\otimes S^{n^{\prime }}(gl(N|N)\otimes
V)^{GL(N|N)}\rightarrow \\
\rightarrow S^{n+n^{\prime }-2}(gl(N|N)\otimes V)^{GL(N|N)}.
\end{multline*}
\end{proposition}

\begin{remark}
The bracket on the subspace $F_{1}$, linearly generated by cyclic words,
coincides with the symplectic bracket from \cite{K}. As explained in \cite%
{B1}, the noncommutative symplectic geometry from \cite{K} can be viewed as
the quasiclassical or, equivalently, tree-level approximation of the
noncommutative Batalin-Vilkovisky geometry described in \cite{B1}.
\end{remark}

The important consequence of the theorem \ref{deltamx} is that the
cohomology of the differential $\Delta $ acting on $F$ is zero. This follows
from the standard identification of the Batalin-Vilkovisky differential on
affine space with the de Rham differential.

\begin{proposition}
The matrix Batalin-Vilkovisky complex $(S(gl(N|N)\otimes V),\Delta )$ is
naturally isomorphic to the de Rham complex of the affine space $%
(gl(N|N)\otimes \Pi V)_{0}$
\end{proposition}

\begin{proof}
Let's identify $S(gl(N|N)\otimes V)$ with algebra of polynomial functions on 
$\limfunc{Hom}(V,gl(N|N))\overset{l}{\simeq }$ $gl(N|N)\otimes \Pi V$. If $%
x_{i}\in V_{0}$, $x_{\pi i}\in V_{1}$ is a basis in which the odd scalar
product has the standard form $l(x_{i},x_{\pi j})=\delta _{ij}$ then, in
terms of the corresponding matrix elements $X_{i,\alpha }^{\beta }$,
generating the algebra of polynomial functions on $\limfunc{Hom}(V,gl(N|N)$,
the isomorphism with the de Rham complex is the standrad isomorphism between
polyvector fields and forms, the "odd Fourier transform". For any set of odd
elements $X_{j,\alpha }^{\beta }$,$\overline{j}+\overline{\alpha }+\overline{%
\beta }=1$, it sends their product to the differential form 
\begin{equation}
\tprod X_{j,\alpha }^{\beta }\leftrightarrow \left( \tprod i(\frac{\partial 
}{\partial X_{\pi j,\beta }^{\alpha }})\right) \Omega   \label{fourodd1}
\end{equation}%
where $X_{\pi j,\beta }^{\alpha }$ are the even elements and%
\begin{equation*}
\Omega =\tprod_{\overline{j}+\overline{\alpha }+\overline{\beta }%
=0}d(X_{j,\alpha }^{\beta })
\end{equation*}%
is the canonical constant volume form on $(gl(N|N)\otimes \Pi V)_{0}$, and $%
i(v)$ is the standard contraction with the vector field $v$.
\end{proof}

\begin{theorem}
The cohomology of the Batalin-Vilkovisky differential acting on $F$ are
trivial: $H^{\ast }(\oplus _{n=0}^{\infty }F_{n},\Delta )=0$ .
\end{theorem}

\begin{proof}
The (algebraic) de Rham complex on the affine space $(gl(N|N)\otimes \Pi
V)_{0}$ has an extra grading $\deg (p(x)\Pi dx^{\alpha })=\deg (p(x))$, so
that, for $N>n+2$, 
\begin{eqnarray*}
F_{n} &\simeq &(\tbigoplus_{i=0}^{n}\Omega _{(gl(N|N)\otimes \Pi
V)_{0}}^{2rN^{2}-i,n-i})^{GL(N|N)}, \\
d_{DR} &:&\Omega _{(gl(N|N)\otimes \Pi V)_{0}}^{j,l}\rightarrow \Omega
_{(gl(N|N)\otimes \Pi V)_{0}}^{j+1,l-1}
\end{eqnarray*}%
since the constant volume form $\Omega $ is $gl(N|N)-$invariant. The
cohomology of the de Rham differential are trivial on every bi-graded piece $%
\Omega _{(gl(N|N)\otimes \Pi V)_{0}}^{2rN^{2}-i,n-i}$. The standard
arguments, see e.g.\cite{GW}, \cite{L}, show that the cohomology are
concentrated on the $GL(N|N)$-invariant subspace. It follows that $\ker
\Delta |_{F_{n}}=$ $\func{im}\Delta |_{F_{n+2}}$.
\end{proof}

Another consequence of the theorem \ref{deltamx} is a version of Morita
equivalence, i.e. the action by tensor multiplication by $(gl(k|\widetilde{k}%
),tr)$ on solutions to the noncommutative BV-equation. Recall, see \cite{B1},%
\cite{B2}, that the noncommutative Batalin-Vilkovisky equation is the
equation 
\begin{equation}
\hbar \Delta S+\frac{1}{2}\{S,S\}=0\Leftrightarrow \Delta (\exp \frac{1}{%
\hbar }S)=0  \label{ncBV}
\end{equation}%
for series of products of cyclic words 
\begin{equation*}
S=\sum_{i,g\geq 0}\hbar ^{2g+i-1}S_{i,g},~~S_{i,g}\in F_{i}
\end{equation*}

\begin{theorem}
Let 
\begin{equation*}
(\widetilde{V},\widetilde{l})=(V,l)\otimes (gl(k|\widetilde{k}),tr)
\end{equation*}%
and let $M:F(V)\rightarrow F(\widetilde{V})$ be the algebra map defined on
generators by%
\begin{equation*}
(a_{\rho _{1}},\ldots ,a_{\rho _{r}})^{c}\rightarrow tr(A_{\rho _{1}},\ldots
,A_{\rho _{r}})^{c},
\end{equation*}%
Then for any solution $S$ to the noncommutative Batalin-Vilkovisky equation
in $F(V)$, $M(S)$ is a solution in $F(\widetilde{V})$. These solutions have
extra $gl(k|\widetilde{k})-$symmetry.
\end{theorem}

\begin{proof}
Notice that as vector spaces with scalar products 
\begin{equation*}
(gl(N|N),tr)\otimes (gl(k|\widetilde{k}),tr)\simeq (gl(\widetilde{N},%
\widetilde{N}),tr)
\end{equation*}%
where $\widetilde{N}=N(k+\widetilde{k})$, and therefore,%
\begin{equation*}
gl(\widetilde{N}|\widetilde{N})\otimes \Pi V\simeq gl(N|N)\otimes \Pi 
\widetilde{V}
\end{equation*}%
as affine spaces with constant BV structures. The $GL(\widetilde{N}|%
\widetilde{N})-$invariant function $\mu _{F(V)}(S)$ on $gl(\widetilde{N}|%
\widetilde{N})\otimes \Pi V$, coincides with $GL(N|N)-$invariant function on 
$gl(N|N)\otimes \Pi \widetilde{V}$ corresponding to $\mu _{F(\widetilde{V}%
)}M(S)$. Therefore $\mu _{F(\widetilde{V})}M(S)$ satisfies the
Batalin-Vilkovisky quantum master equation and so does $M(S)$.
\end{proof}

\section{Modular operad structure on $k[S_{n}]$.}

In \textbf{\ }\cite{B1} the operations on collection of spaces $%
\{k[S_{n}]\}_{n\geq 1}$ giving rise to the modular operad structure were
defined. The subspace of cyclic permutations corresponds to the cyclic
operad of associative algebras with\emph{\ }scalar product. The relation
with $GL(U)-$invariant tensors on the matrix spaces allows to give a
straightforward definition for this modular operad structure.

I work in the category of $\mathbb{Z}/2\mathbb{Z}$-graded vector spaces and
the modification of the modular operad notion needed here is defined as the
algebra over triple, which is the functor on $\mathbb{S}-$modules given by 
\begin{equation*}
\mathbb{M}\mathcal{V}((n))=\bigoplus_{G\in \Gamma ((n))}\mathcal{V}%
((G))_{Aut(G)}
\end{equation*}%
i.e. forgetting the extra $\mathbb{Z}-$grading, compared with definition
from \cite{GK}. It is straightforward to see that forgetting the extra $%
\mathbb{Z}-$grading and orientation on the spaces of cycles, the formulas
from (\cite{B1}, section 9) define such modular operad, which I denote also
by $\mathcal{S}$ as in loc.cit.

Consider the endomorphism modular operad $\mathcal{E}[End_{k}(U)]$,
associated with the vector space $End_{k}(U)$, $\dim _{k}U=(N|N)$, equipped
with the even scalar product defined by the super trace (\ref{trEnd}). I
have 
\begin{equation*}
\mathcal{E}[End_{k}(U)]((n))=End_{k}(U)^{\otimes n}
\end{equation*}%
and contractions along graphs are defined via contractions with the
two-tensor corresponding to the super trace. The structure maps of $\mathcal{%
E}[End_{k}(U)]$ are invariant under the $GL(U)$-action. Consider the $GL(U)-$%
invariant modular suboperad $\mathcal{E}[End_{k}(U)]^{GL(U)}$. Because of (%
\ref{mukasn}) its components for$\,\,n<N$ are the same as the components of
the operad $\mathcal{S}((n))=k[S_{n}]$ 
\begin{equation*}
\mathcal{E}[End_{k}(U)]((n))^{GL(U)}\simeq \mathcal{S}((n)).
\end{equation*}%
For the space $U^{\prime }=U\oplus k^{1|1}$, the natural maps $\mathcal{E}%
[End_{k}(U)]((n))^{GL(U)}\rightarrow \mathcal{E}[End_{k}(U^{\prime
})]((n))^{GL(U^{\prime })}$ are isomorphisms for small $\,n<N$. Consider the
modular operad $\mathcal{E}[End_{k}]^{GL}$ which is the direct limit of $%
\mathcal{E}[End_{k}(U_{i})]^{GL(U_{i})}$, $\dim _{k}U_{i}=(N_{i}|N_{i})$, $%
N_{i}\rightarrow \infty $:%
\begin{equation*}
\mathcal{E}[End_{k}]^{GL}=\lim_{\rightarrow }\mathcal{E}%
[End_{k}(U_{i})]^{GL(U_{i})}
\end{equation*}

Recall, see (\cite{B1}, section 9), that the basic contraction operators 
\begin{equation*}
\mu _{ff^{\prime }}^{\mathcal{S}}:\mathcal{S}((I\sqcup \{f,f^{\prime
}\}))\rightarrow \mathcal{S}((I))
\end{equation*}%
are defined for the modular operad $\mathcal{S}$ as the linear maps%
\begin{equation*}
k[Aut(I\sqcup \{f,f^{\prime }\})]\rightarrow k[Aut(I)]
\end{equation*}%
which act on permutations of the set $(I\sqcup \{f,f^{\prime }\})$ via 
\begin{eqnarray*}
(\rho _{1}\ldots \rho _{p-1}f\rho _{p+1}\ldots \rho _{r})\ldots (\tau
_{1}\ldots \tau _{q-1}f^{\prime }\tau _{q+1}\ldots \tau _{t}) &\rightarrow &
\\
\rightarrow (\rho _{1}\ldots \rho _{p-1}\tau _{q+1}\ldots \tau _{t}\tau
_{1}\ldots \tau _{q-1}\rho _{p+1\ldots }\rho _{r})\ldots &&
\end{eqnarray*}%
if the elements $f$ and $f^{\prime }$ are in the different cycles of the
permutation, and via%
\begin{eqnarray}
(\rho _{1}\ldots \rho _{p-1}f\rho _{p+1}\ldots \rho _{q-1}f^{\prime }\rho
_{q+1}\ldots \rho _{r})\ldots (\tau _{1}\ldots \tau _{t}) &\rightarrow &
\label{rofrof} \\
\rightarrow (\rho _{1}\ldots \rho _{p-1}\rho _{q+1}\ldots \rho _{r})(\rho
_{p+1}\ldots \rho _{q-1})\ldots (\tau _{1}\ldots \tau _{t}) &&  \notag
\end{eqnarray}%
\begin{equation}
(\rho _{1}\ldots \rho _{p-1}ff^{\prime }\rho _{p+1}\ldots \rho _{r})\ldots
(\tau _{1}\ldots \tau _{t})\rightarrow 0  \label{roff}
\end{equation}%
if the elements $f$ and $f^{\prime }$ are in the same cycle of the
permutation.

\begin{proposition}
The modular operad $\mathcal{S}$ is isomorphic to the modular operad $%
\mathcal{E}[End_{k}]^{GL}$
\end{proposition}

\begin{proof}
The calculations are very similar to the calculations from the proof of the
theorem \ref{deltamx}. In particular the condition (\ref{uevenuodd}) implies
(\ref{roff}).
\end{proof}

\section{\protect\bigskip Even scalar product.}

In the case of even scalar product the quantum master equation of the
noncommutative Batalin-Vilkovisky geometry is defined on the space 
\begin{equation*}
F=Symm(\oplus _{j=0}^{\infty }\Pi ((\Pi V)^{\otimes j})^{\mathbb{Z}/j\mathbb{%
Z}})
\end{equation*}%
with components 
\begin{equation*}
F_{n}=((\Pi V)^{\otimes n}\otimes k[\mathbb{S}_{n}]^{\prime })^{\mathbb{S}%
_{n}}
\end{equation*}%
where $k[\mathbb{S}_{n}]^{\prime }$ is the vector space with the basis
indexed by elements $(\sigma ,\rho _{\sigma })$, where $\sigma \in \mathbb{S}%
_{n}$ is a permutation with $i_{\sigma }$ cycles $\sigma _{\alpha }$ and $%
\rho _{\sigma }=\sigma _{1}\wedge \ldots \wedge \sigma _{i_{\sigma }}$, $%
\rho _{\sigma }\in Det(Cycle(\sigma ))$, $Det(Cycle(\sigma
))=Symm^{i_{\sigma }}(k^{0|i_{\sigma }})$, is one of the generators of the
one-dimensional determinant of the set of cycles of $\sigma $, i.e. $\rho
_{\sigma }$ is an order on the set of cycles defined up to even reordering,
and $(\sigma ,-\rho _{\sigma })=-(\sigma ,\rho _{\sigma })$.

If an even scalar product is fixed on the space $V$ , then $F$ has canonical
differential $\Delta $, and it defines the Batalin-Vilkovisky algebra
structure on $F$, see \cite{B1},\cite{B2}.

Consider again the $\mathbb{Z}/2\mathbb{Z}$-graded vector space $U$, $\dim
U_{0}=\dim U_{1}=N$. Let $p$, $p^{2}=1$, denotes an \emph{odd} involution
acting on $U$. It acts by interchanging isomorphically $U_{0}$ with $U_{1}$.
The Bernstein-Leites algebra is the subalgebra of $End(U)$ of operators
commuting with $p$:%
\begin{equation*}
q(U)=\{G\in End(U)|\,[G,p]=0\}.
\end{equation*}%
It looks as follows in the standard block decomposition of supermatrices:%
\begin{equation}
G=\left( 
\begin{array}{cc}
X & Y \\ 
-Y & X%
\end{array}%
\right)   \label{GXdzeta}
\end{equation}%
in the base in which $p=\left( 
\begin{array}{cc}
0 & 1_{N} \\ 
1_{N} & 0%
\end{array}%
\right) $. As a $\mathbb{Z}/2\mathbb{Z}$-graded vector space $q(U)$ is
isomorphic to $\Pi TEnd(U_{0})$: 
\begin{equation*}
q(U)=End(U_{0})\oplus \Pi End(U_{0}).
\end{equation*}%
The algebra structure is different however from the standard multiplication
on $End(U_{0})\otimes k[\xi ]/\{\xi ^{2}=0\}$. The algebra $q(U)$ is
isomorpic to the tensor product of $End(U_{0})$ with the Clifford algebra $%
Cl(1)$: 
\begin{equation*}
q(U)=End(U_{0})\otimes Cl(1),Cl(1)=k[\xi ]/\{\xi ^{2}=1\}
\end{equation*}%
The property of $q(U)$ of the main interest here is that $q(U)$ has an \emph{%
odd }analog of the super trace functional:%
\begin{gather*}
otr(G)=\frac{1}{2}tr(Gp)=(-1)^{\overline{G}}\frac{1}{2}tr(pG)=trY, \\
otr([G,G^{\prime }])=0
\end{gather*}%
which gives canonical \emph{odd }invariant scalar product on $q(U)$:%
\begin{equation*}
(G,G^{\prime })\rightarrow otr(GG^{\prime })=tr(XY^{\prime })+tr(YX^{\prime
}).
\end{equation*}%
This odd scalar product on $q(U)$ together with even scalar product on $V$
defines the natural odd symmetric scalar product on the tensor product $%
\mathbb{Z}/2\mathbb{Z}$-graded vector space%
\begin{equation*}
\limfunc{Hom}(q(U),\Pi V)
\end{equation*}%
Therefore, as in the previoius case, this space is an affine space with
constant odd symplectic structure and therefore, its algebra of symmetric
tensors 
\begin{equation*}
\oplus _{n=0}^{\infty }S^{n}\limfunc{Hom}(q(U),\Pi V)
\end{equation*}%
has natural structure of Batalin-Vilkovisky algebra.

The subgroup $GQ(U)\subset GL(U)$, preserving the odd involution $p$:%
\begin{equation*}
GQ(U)=\{g\in GL(U)|gpg^{-1}=p\}.
\end{equation*}%
is the super group, which is acts on the Lie algebra $q(U)$ via the adjoint
representation. The supergroup $GQ$ is an odd analog of the general linear
group $GL$. In order to desctribe the quantum master equation of
noncommutative Batalin-Vilkovisky geometry in the even scalar product case
I've introduced in \cite{B1} the twisted group algebras $k[\mathbb{S}%
_{n}]^{\prime }$. Next proposition shows that, as it follows from the
results of \cite{S}, taking invariants in the tensor powers of the coadjoint
representation of $q(U)$ gives precisely this twisted group algebras $k[%
\mathbb{S}_{n}]^{\prime }$, in complete analogy with the result for $gl(U)$.
The description of $q(U)-$invariants in terms of products of odd traces
based on this seems to be new.

\begin{proposition}
\label{ksnpr} 
\begin{equation}
\limfunc{Hom}(q(U)^{\otimes n},k)^{GQ(U)}=k[\mathbb{S}_{n}]^{\prime }
\label{qksnpr}
\end{equation}%
for $\dim U=(N|N)$ sufficiently big $(N>n)$. The $GQ(U)-$invariants in $%
\limfunc{Hom}(q(U)^{\otimes n},k)$ are spanned linearly by the products of
odd traces:%
\begin{equation*}
otr(A_{\rho _{1}}\ldots A_{\rho _{r}})\cdot \ldots \cdot otr(A_{\tau
_{1}}\ldots A_{\tau _{t}}),\,\,\,\,A_{i}\in q(U).
\end{equation*}
\end{proposition}

\begin{proof}
See \cite{S} for the proof of the first statement. Here is a sketch of
alternative argument. The basis for invariants of the $GQ(U)$ action on $%
\limfunc{Hom}(End(U),k)$ is $tr(G)=\sum_{a}e_{a}\otimes e^{a}$ and $%
tr(pG)=\sum_{a}(pe_{a})\otimes e^{a}$. Therefore on the space of tensors $%
\limfunc{Hom}(End(U),k)^{\otimes n}$ the space of $GQ(U)-$invariants is
spanned by all possible combinations of these two elements of the form 
\begin{equation*}
\sum_{a_{1},\ldots ,a_{n}}(e_{a_{\sigma (1)}}\otimes e^{a_{1}})\otimes
\ldots \otimes (pe_{a_{\sigma (i)}}\otimes e^{a_{i}})\otimes \ldots ,
\end{equation*}%
where $\sigma \in \mathbb{S}_{n}$. Such element corresponds to an arbitrary
permutation $\sigma \in \mathbb{S}_{n}$ and the marking which associates one
of the two types of tensors $(e_{a_{\sigma (i)}}\otimes e^{a_{i}})$ or $%
(pe_{a_{\sigma (i)}}\otimes e^{a_{i}})$ to every $i\in \{1,\ldots ,n\}$. If
the cycle decomposition of $\sigma $ is denoted by $(\rho _{1}\ldots \rho
_{r})\ldots (\tau _{1}\ldots \tau _{t})$, then such an element gives the
linear functional on $End(U)^{\otimes n}$ of the following type:%
\begin{equation*}
tr(A_{\rho _{1}}\ldots A_{\rho _{i-1}}pA_{\rho _{i}}\ldots A_{\rho
_{r}})\cdot \ldots \cdot tr(A_{\tau _{1}}\ldots A_{\tau _{j-1}}pA_{\tau
_{j}}\ldots A_{\tau _{t}})
\end{equation*}%
consisting of products of traces of compositions of the endomorphisms with
arbitrary inclusions of the operator $p$. The subspace $q(U)\subset End(U)$
has complementary subspace, preserved by $GQ(U)$, which consists of
endomorphisms anticommuting with $p$. Therefore, the restriction map from $%
\limfunc{Hom}(End(U)^{\otimes n},k)^{GQ(U)}$ to $\limfunc{Hom}(q(U)^{\otimes
n},k)^{GQ(U)}$ is onto. The operator $p$ commutes with any $A_{\rho _{j}}\in
q(U)$. Therefore on $q(U)^{\otimes n}$ all inclusions of $p$ inside the
given trace cancel with each other, except for possibly one inclusion: 
\begin{equation*}
tr(A_{\rho _{1}}\ldots A_{\rho _{i-1}}pA_{\rho _{i}}\ldots A_{\rho
_{r}})=tr(p^{l}A_{\rho _{1}}\ldots A_{\rho _{r}})
\end{equation*}%
where $l=0$ or $l=1$ depending on the parity of the total number of
inclusions of $p$. Notice now that for any $A\in q(U)$ , $tr(A)=0$. And
therefore, the traces with even number of inclusions of $p$ vanish on $%
q(U)^{\otimes n}$. The trace with odd number of inclusions of $p$ becomes
the odd trace $otr$ when restricted to $q(U)$ . Therefore the $GQ(U)-$%
invariants in $\limfunc{Hom}(q(U)^{\otimes n},k)$ are spanned by the
products of odd traces:%
\begin{equation*}
otr(A_{\rho _{1}}\ldots A_{\rho _{r}})\cdot \ldots \cdot otr(A_{\tau
_{1}}\ldots A_{\tau _{t}}),\,\,\,\,A_{i}\in q(U).
\end{equation*}%
One can also deduce from the corresponding result for $gl$, that these
products of odd traces are linearly independent for $N\geq n$.
\end{proof}

The super group $GQ(U)$ preserves the odd trace $otr$ and therefore, the
invariants subspace $\oplus _{n=1}^{\infty }(S^{n}$ $\limfunc{Hom}(q(U),\Pi
V))^{GQ(U)}$ inherits the natural Batalin-Vilkovisky algebra structure. I
have now the following analogs of the propositions \ref{propFn}, \ref%
{propbrackt} and of the theorem \ref{deltamx}. The proofs are completely
analogous to the proofs in the odd scalar product case.

\begin{proposition}
The vector space $F_{n}$ is canonically identified with $GQ(U)$- invariant
subspace of $n-$th symmetric powers of the vector space $q(U)\otimes V$: 
\begin{equation}
F_{n}\simeq \left( S^{n}\limfunc{Hom}(q(U),\Pi V)\right) ^{GQ(U)},
\label{thetainv}
\end{equation}%
where $q(U)\subset GL(U)$ is the odd general linear algebra and $\dim
_{k}U=N $, $N\geq n$.
\end{proposition}

Identify the symmetric algebra generated by $\limfunc{Hom}(q(U),\Pi V)$ with
polynomial functions on $\limfunc{Hom}(\Pi V,q(U)).$ For an element $%
a_{i}\in \Pi V$ let $A_{i}$ denotes the corresponding $q(U)-$valued linear
function on $\limfunc{Hom}(\Pi V,q(U))$.

\begin{proposition}
The isomorphism (\ref{thetainv}) sends $(a_{\rho _{1}}\ldots a_{\rho
_{r}})^{c}\wedge \ldots \wedge (a_{\tau _{1}}\ldots a_{\tau _{t}})^{c}$ to
the product of odd traces%
\begin{equation*}
otr(A_{\rho _{1}}\ldots A_{\rho _{r}})\cdot \ldots \cdot otr(A_{\tau
_{1}}\ldots A_{\tau _{t}})
\end{equation*}
\end{proposition}

\begin{theorem}
The operator $\Delta $ defined on the $GQ(U)-$invariant subspace 
\begin{equation*}
\Delta :\left( S^{n}\limfunc{Hom}(q(U),\Pi V)\right) ^{GQ(U)}\rightarrow
\left( S^{n-2}\limfunc{Hom}(q(U),\Pi V)\right) ^{GQ(U)},
\end{equation*}%
where $\dim _{k}U=(N|N)$ is sufficiently big ($N\geq n$), coincides with the
differential $\Delta :$ $F_{n}\rightarrow F_{n-2}$ from \cite{B1},\cite{B2}.
\end{theorem}

\begin{proposition}
\bigskip The odd symplectic bracket on the $GQ(U)-$invariant subspaces%
\begin{multline*}
(S^{n}\limfunc{Hom}(q(U),\Pi V))^{GQ(U)}\otimes (S^{n^{\prime }}\limfunc{Hom}%
(q(U),\Pi V))^{GQ(U)}\rightarrow \\
\rightarrow S^{n+n^{\prime }-2}\limfunc{Hom}(q(U),\Pi V)^{GQ(U)}
\end{multline*}%
coincides with the standard odd symplectic bracket (see references in
loc.cit): 
\begin{equation*}
F_{n}\otimes F_{n^{\prime }}\rightarrow F_{n+n^{\prime }-2}.
\end{equation*}
\end{proposition}

\bigskip Next in order to calculate the cohomology of $\Delta $ on $S(%
\limfunc{Hom}(q(U),\Pi V)$ I identify it with the de Rham complex of the
affine space $(q(U)\otimes \Pi V)_{0}$. For $U$ with a fixed basis $U\simeq
k^{N}$ I use the standard notation $q(N)$ .

\begin{proposition}
The matrix Batalin-Vilkovisky complex $(S(\limfunc{Hom}(q(N),\Pi V)),\Delta
) $ is naturally isomorphic to the de Rham complex of the affine space $%
(q(N)\otimes \Pi V)_{0}$
\end{proposition}

\begin{proof}
As above I consider $S(\limfunc{Hom}(q(N),\Pi V))$ as the algebra of
polynomial functions on the affine space $\limfunc{Hom}(\Pi V,q(N))$. Let $%
v_{j}\in \Pi V$ be a basis and let $(X_{j,\alpha }^{\beta },Y_{j,\alpha
}^{\beta })$ denotes the corresponding natural basis of linear functions on $%
\limfunc{Hom}(\Pi V,q(N))$. Denote by $\widetilde{v}_{j}$ the dual basis
satisfying $l(\pi y_{j},\pi \widetilde{y}_{j^{\prime }})=\delta _{jj^{\prime
}}$ and via $(\widetilde{X}_{j,\alpha }^{\beta },\widetilde{Y}_{j,\alpha
}^{\beta })$ the corresponding basis of linear functions on $\limfunc{Hom}%
(\Pi V,q(N))$. The standard isomorphism between polyvector fields and forms
sends in this case the product of a set of odd elements $X_{j,\alpha
}^{\beta }$, $\overline{j}=1$, $Y_{k,\varepsilon }^{\gamma }$, $\overline{k}%
=0$, to the differential form 
\begin{equation}
\tprod X_{j,\alpha }^{\beta }\tprod Y_{k,\varepsilon }^{\gamma
}\leftrightarrow \left( \tprod i(\frac{\partial }{\partial \widetilde{Y}%
_{j,\beta }^{\alpha }})\tprod i(\frac{\partial }{\partial \widetilde{X}%
_{k,\varepsilon }^{\gamma }})\right) \Omega  \label{fourodd2}
\end{equation}%
where 
\begin{equation*}
\Omega =\tprod_{\overline{j}=0,\alpha ,\beta }d(X_{j,\alpha }^{\beta
})\tprod_{\overline{j}=1,\alpha ,\beta }d(Y_{j,\alpha }^{\beta })
\end{equation*}%
is the constant volume form on $(q(N)\otimes \Pi V)_{0}$ and $i(v)$ is the
standard contraction with the vector field $v$.
\end{proof}

\begin{theorem}
The cohomology of the Batalin-Vilkovisky differential acting on $F$ are
trivial: $H^{\ast }(\oplus _{n=0}^{\infty }F_{n},\Delta )=0$ .
\end{theorem}

\begin{proof}
The proof is parallel to the case of odd scalar product above and follows
from the identification 
\begin{eqnarray*}
F_{n} &\simeq &(\tbigoplus_{i=0}^{n}\Omega _{(q(N)\otimes \Pi
V)_{0}}^{rN^{2}-i,n-i})^{GQ(N)}, \\
d_{DR} &:&\Omega _{(q(N)\otimes \Pi V)_{0}}^{j,l}\rightarrow \Omega
_{(q(N)\otimes \Pi V)_{0}}^{j+1,l-1}
\end{eqnarray*}%
where $\Omega _{(q(N)\otimes \Pi V)_{0}}^{j,l}$ is the de Rham complex with
the extra grading by the polynomial degree, $N>n+2$ and $r=\dim _{k}V$. The
cohomology of the de Rham differential are trivial on every bi-graded piece $%
\Omega _{(q(N)\otimes \Pi V)_{0}}^{rN^{2}-i,n-i}$ and the standard
arguments, see e.g.\cite{GW}, \cite{L}, imply that the cohomology are
concentrated on the $GQ(N)$-invariant subspace. It follows that $\ker \Delta
|_{F_{n}}=$ $\func{im}\Delta |_{F_{n++2}}$.
\end{proof}

\bigskip A version of Morita equivalence holds also.

\begin{theorem}
Let 
\begin{equation*}
(\widetilde{V},\widetilde{l})=(V,l)\otimes (gl(k|k^{\prime }),tr)
\end{equation*}%
and let $M:F(V)\rightarrow F(\widetilde{V})$ be the map of the commutative
algebras defined on generators by%
\begin{equation*}
(a_{\rho _{1}},\ldots ,a_{\rho _{r}})^{c}\rightarrow tr(A_{\rho _{1}},\ldots
,A_{\rho _{r}})^{c}
\end{equation*}%
Then for any solution $S$ to the noncommutative Batalin-Vilkovisky equation
in $F(V)$, $M(S)$ is a solution in $F(\widetilde{V})$. Class of such
solutions is characterized by the extra $gl(k|k^{\prime })-$symmetry.
\end{theorem}

\begin{proof}
This follows from the isomorphism of vector spaces with scalar products 
\begin{equation*}
(q(N),otr)\otimes (gl(k|k^{\prime }),tr)\simeq (q(\widetilde{N}),otr)
\end{equation*}%
where $\widetilde{N}=N(k+k^{\prime })$, which follows directly from the
definition of $q(N)$ and $otr$. Therefore%
\begin{equation*}
q(\widetilde{N})\otimes \Pi V\simeq q(N)\otimes \Pi \widetilde{V}
\end{equation*}%
as affine spaces with constant BV structures. The $GQ(\widetilde{N})-$%
invariant function $\mu _{F(V)}(S)$ on $q(\widetilde{N})\otimes \Pi V$,
coincides by definition with $GQ(N)-$invariant function on $q(N)\otimes \Pi 
\widetilde{V}$ corresponding to $\mu _{F(\widetilde{V})}M(S)$. Hence $\mu
_{F(\widetilde{V})}M(S),$and therefore $M(S)$, both satisfy the
Batalin-Vilkovisky quantum master equations .
\end{proof}

\subsection{\protect\bigskip Super Morita equivalence.}

It is remarkable that thanks to the supersymmetry there exists a
superversion of Morita equivalence.

\begin{theorem}
Let $(V,l)$ be a vector space with odd (even) scalar product, define the
vector space with \emph{even (}respectively,\emph{\ odd)} scalar product 
\begin{equation*}
(\widetilde{V},\widetilde{l})=(V,l)\otimes (q(1),otr)
\end{equation*}%
and let $M:F(V)\rightarrow F(\widetilde{V})$ be the map of the commutative
algebras defined on generators by%
\begin{equation*}
(a_{\rho _{1}},\ldots ,a_{\rho _{r}})^{c}\rightarrow otr(A_{\rho
_{1}},\ldots ,A_{\rho _{r}})^{c},
\end{equation*}%
Then for any solution $S$ to the noncommutative Batalin-Vilkovisky equation
in $F(V)$, $M(S)$ is a solution in $F(\widetilde{V})$. These solutions have
extra $GQ(1)-$symmetry.
\end{theorem}

\begin{proof}
Let $(V,l)$ be a vector space with odd scalar product, the other case is
analogous. The statement follows from the isomorphism of vector spaces with
scalar products 
\begin{equation*}
(q(N),otr)\otimes (q(1),otr)\simeq (gl(N|N),tr)
\end{equation*}%
which follows from the definition of $q(N)$ and $otr$. Therefore%
\begin{equation}
gl(N|N)\otimes \Pi V\simeq q(N)\otimes \Pi \widetilde{V}  \label{glnqn}
\end{equation}%
as affine spaces with constant BV structures. The $GL(N|N)-$invariant
function $\mu _{F(V)}(S)$ on $gl(N|N)\otimes \Pi V$, coincides by definition
with $GQ(N)-$invariant function $\mu _{F(\widetilde{V})}M(S)$ after
identification (\ref{glnqn}). Hence $\mu _{F(\widetilde{V})}M(S),$and
therefore, $M(S)$, both satisfy the Batalin-Vilkovisky quantum master
equations .
\end{proof}

\subsection{Twisted modular operad on $k[\mathbb{S}_{n}]^{\prime }$.}

Using (\ref{qksnpr}) it is straightforward to prove the following statement

\begin{proposition}
The twisted modular operad structure on the collection of spaces $\{k[%
\mathbb{S}_{n}]^{\prime }\}$ described in \cite{B1} is isomorpic to the
stable part, as $N\rightarrow \infty $, of the $GQ-$invariants suboperad of
the standard\emph{\ tensor} twisted modular operad, based on the vector
space with odd scalar product $(q(N),otr)$%
\begin{equation*}
\mathcal{E}[q(N)]((n))=q(N)^{\otimes n}
\end{equation*}%
with contractions along graphs defined via contractions with the two-tensor
corresponding to the odd trace.
\end{proposition}

\section{\protect\bigskip Equivariant differential and localization.\label%
{secEqDif}}

In the previous sections I have identified the non-commutative
Batalin-Vilkovisky operator $\Delta $, acting on the basic spaces $F$, with
de Rham differntial acting on $GL(N|N)$ and $GQ(N)$ invariant subspaces in
the de Rham complexes of affine spaces. The invariant suspaces of de Rham
complexes appear also in the definition of the equivariant differential. So
it is natural to look for the analogue of the remaining part of the
equivariant differential acting on $F$. I'll treat both even and odd cases
simultaneously in this and the next sections. Denote by $g$ and $G$ in the
odd scalar product case the super Lie algebra $gl(N|N)$ and $GL(N|N)$, and
in the even scalar product case the super Lie algebra $q(N)$ and the super
group $GQ(N)$. Denote by $\left\langle \cdot ,\cdot \right\rangle $ the odd
symplectic structure on the affine space $g\otimes \Pi V$.

\begin{proposition}
The adjoint action of the super Lie algebra $g$ on the vector space $%
g\otimes \Pi V$ preserves the odd symplectic structure. The Hamiltonian of
the linear vector field corresponding to $\gamma \in g$, is the quadratic
function 
\begin{equation}
S_{2,\gamma }=\left\langle [\gamma ,X],X\right\rangle  \label{qham}
\end{equation}%
It satisfies \newline
\begin{equation}
\Delta S_{2,\gamma }=0  \label{dltasgamma}
\end{equation}
\end{proposition}

\begin{proof}
The first equation is the standard formula of symplectic geometry and
follows from the definition of the hamiltonian: 
\begin{equation*}
dS_{2,\gamma }(Y)=(i_{ad(\gamma )}\left\langle \cdot ,\cdot \right\rangle
)(Y)
\end{equation*}%
for any vector field $Y$. The second formula is equivalent to 
\begin{equation*}
trace(ad(\gamma )|_{g})=0
\end{equation*}%
satisfied for any $\gamma $ from the algebras $gl(N|N)$ and $q(N)$.
\end{proof}

\begin{proposition}
\label{Ladgamma}The Lie derivative by the linear vector fields $L_{ad(\gamma
)}$ satisfies on $\mathcal{O}(g\otimes \Pi V)$ the Cartan homotopy formula,%
\begin{equation*}
L_{ad(\gamma )}=[\Delta ,S_{2,\gamma }]
\end{equation*}%
where, slightly abusing notation, I denote by $S_{2,\gamma }$ the
multiplication by the quadratic hamiltonian (\ref{qham}).
\end{proposition}

\begin{proof}
This is immedaite from (\ref{dltasgamma}) and the basic formula of
Batalin-Vilkovisky geometry, expressing the bracket via the action of $%
\Delta $.
\end{proof}

\begin{proposition}
The equivariant differential $d_{DR}+i_{\gamma }$ of the action of $G_{0}$
on $\Omega _{DR}((g\otimes \Pi V)_{0})$ corresponds to the differential on $%
\mathcal{O}(g\otimes \Pi V)$:%
\begin{equation*}
\Delta +S_{2,\gamma }
\end{equation*}
\end{proposition}

In the next proposition I show that promoting the standard action of $g_{0}$
on $\Omega _{DR}((g\otimes \Pi V)_{0})$ to the action of the super Lie
algebra $g$ gives natural construction of equivariantly closed differential
forms for the action of one-parameter subgroups of $G_{0}$ generated by
elements of the form $[\gamma ,\gamma ]$, $\gamma \in g_{1}$.

\begin{proposition}
For any $G-$ invariant function $\Psi \in \mathcal{O}(g\otimes \Pi V)^{G}$
the formula 
\begin{equation*}
\gamma \in g\rightarrow \exp (S_{2,\gamma })\Psi
\end{equation*}%
defines a $G-$ invariant element from $(\mathcal{O}(g)\otimes \mathcal{O}%
(g\otimes \Pi V))^{G}$, which corresponds under the odd Fourier transform to
the equivariant differential form from $(\mathcal{O}(g)\otimes \Omega
_{DR}((g\otimes \Pi V)_{0}))^{G}$. If $\Psi $ is $\Delta -$closed, then this
element is closed under the equivariant differential 
\begin{equation*}
(\Delta +S_{2,\zeta })\left[ \exp (S_{2,\gamma })\Psi \right] =0
\end{equation*}%
where $\zeta =\frac{1}{2}[\gamma ,\gamma ]$
\end{proposition}

\begin{proof}
By proposition \ref{Ladgamma} for any $G-$ invariant function 
\begin{equation*}
\lbrack \Delta ,S_{2,\gamma }]\Psi =0.
\end{equation*}%
Now the proof follows from the standard Cartan calculus formula:%
\begin{equation*}
\Delta \left[ \exp (S_{2,\gamma })\right] =\exp (S_{2,\gamma })(\Delta
+[\Delta ,S_{2,\gamma }]+\frac{1}{2}S_{2,[\gamma ,\gamma ]})
\end{equation*}
\end{proof}

\begin{corollary}
The Lagrangians of the matrix integrals 
\begin{equation}
\int_{L}\exp \frac{1}{\hbar }\left( -\frac{1}{2}\left\langle [\Xi
,X],X\right\rangle +S_{q}(X)\right) dX  \label{lagr1}
\end{equation}%
and 
\begin{equation}
\int_{L}\exp \frac{1}{\hbar }\left( -\frac{1}{2}\left\langle [\Xi
,X],X\right\rangle +S_{gl}(X)\right) dX  \label{lagr2}
\end{equation}%
constructed in \cite{B2}, represent, after the odd Fourier transforms (\ref%
{fourodd1},\ref{fourodd2}), equivariant differential forms from $\Omega
_{DR}((g\otimes \Pi V)_{0})$. These forms are closed under the equivariant
differential $(\Delta +S_{2,\Xi ^{2}})=d_{DR}+i_{[\Xi ^{2},\cdot ]}$
\end{corollary}

\begin{example}
In the case of the matrix integral (\ref{lagr2}), with $\Xi =\left( 
\begin{array}{cc}
0 & Id \\ 
\Lambda _{01} & 0%
\end{array}%
\right) $, $\Lambda _{01}\in gl(N)$, this gives an equivariant differential
form from $\Omega _{DR}((g\otimes \Pi V)_{0})$ with respect to the
block-diagonal action$\left( 
\begin{array}{cc}
\Lambda _{01} & 0 \\ 
0 & \Lambda _{01}%
\end{array}%
\right) $ of $gl(N)\subset gl(N|N)$.
\end{example}

\begin{remark}
\bigskip The localization in equivariant cohomology reduces the integrals (%
\ref{lagr1},\ref{lagr2}) over $S^{1}-$invariant relative cycles from $%
H_{\ast }^{S^{1}}(L,\func{Re}\frac{1}{\hbar }(S_{2}+S)<-r)$, $r\rightarrow
+\infty $, where $L$ is the $[\Xi ^{2},\cdot ]-$invariant subspace of $%
(g\otimes \Pi V)_{0}$ and $S^{1}$is the compact subgroup of the group
generated by $[\Xi ^{2},\cdot ]-$action , to the integrals over $[\Xi
^{2},\cdot ]-$fixed points. The details of the computation will appear in 
\cite{B6}.
\end{remark}

\section{Noncommutative AKSZ formalism.\label{secNCaksz}}

The relation of nc-BV differential $\Delta $ with invariant integration with
respect to the supergroups $GQ(N)$ and $GL(N|N)$, gives an interpretation to
the Lagrangians from \cite{B2} as non-commutative \emph{super}-equivariant
analogues of the AKSZ $\sigma -$model. This and other non-commutative
analogues of some standard Lagrangians are studied in \cite{B3} and \cite{B6}%
.

I consider,  for definiteness, the even scalar product case, the case of the
odd scalar product is parallel. The initial step is to interpret the space $%
q(N)\otimes \Pi V$ as the space of morphisms%
\begin{equation*}
Free(\Pi V)^{dual}\rightarrow q(N)
\end{equation*}%
from the free associative algebra generated by $(\Pi V)^{dual}$. This space
can be interpreted as the functor of points of $Spec\left( Free(\Pi
V)^{dual}\right) $ over simple associative \emph{super} algebra $q(N)$. 
\begin{equation*}
Spec(q(N))\rightarrow Spec\left( Free(\Pi V)^{dual}\right) 
\end{equation*}%
Next, interpret the even scalar product on $V$ as an even symplectic 2-form
on $Spec\left( Free(\Pi V)^{dual}\right) $, and interpret the odd trace $otr$
on $q(N)$ as a kind of integral with respect to the odd volume element.
Therefore their tensor product, defining the odd symplectic structure on $%
q(N)\otimes \Pi V$, is the analogue of the odd symplectic structure on space
of maps $f:\Sigma \rightarrow X$:%
\begin{equation*}
(u,v)_{f}=\dint\limits_{\Sigma }\omega _{f(y)}(u(y),v(y))dY
\end{equation*}%
where $u,v\in T_{f}Maps(\Sigma ,X)$. Now notice that the supergroup $GQ(N)$
acts on the space of morphisms $Mor(Free(\Pi V)^{dual},q(N))$, preserving
the odd symplectic structure. The Hamiltonians of the corresponding vector
fields are the quadratic functions $S_{2,\gamma }$ (\ref{qham}). Solution $S$
to the non-commutative BV-equation gives $GQ(N)-$invariant function $\mu
_{F}(S)$ (denoted by $S_{q}$ in (\ref{lagr2})) on 
\begin{equation*}
Mor(Free(\Pi V)^{dual},q(N))
\end{equation*}%
which corresponds also to a hamiltonian vector field. Their sum $%
S_{q}+S_{2,\gamma }$ satisfies the \emph{equivariant} quantum master
equation 
\begin{equation*}
\hbar \Delta (S_{q}+S_{2,\gamma })+\frac{1}{2}[S_{q}+S_{2,\gamma
},S_{q}+S_{2,\gamma }]+S_{2,\frac{1}{2}[\gamma ,\gamma ]}=0
\end{equation*}%
Derivations of $Free(\Pi V)^{dual}$, preserving the even scalar product on $V
$, correspond to cyclic Hochschild cochains. They also act on $Mor(Free(\Pi
V)^{dual},q_{N})$, preserving the odd symplectic structure and commuting
with the the $GQ(N)$ supergroup action. Derivations preserving $S_{q}$
correspond to closed cyclic cochains. Their Hamiltonians can also be added
to the Lagrangian and the resulting integrals depend in addition on the
extra parameters given by the cyclic cohomology classes. The cyclic
cohomology classes can also be viewed as natural observables of the theory.
Such matrix integrals are studied further in \cite{B6}.

\end{document}